\documentclass[a4paper,10pt]{amsart}

\usepackage[utf8]{inputenc}
\usepackage[usenames, dvipsnames]{color}
\newtheorem{theorem}{Theorem}

\newtheorem{lemma}[theorem]{Lemma}
\newtheorem{corollary}[theorem]{Corollary}

\newcommand{\C}{\mathbb C}
\newcommand{\R}{\mathbb R}

\newcommand{\D}{\mathbb D}
\newcommand{\B}{\mathbb B}

\usepackage{ulem}

\title[]{Asymptotic first boundary value problem for elliptic operators}

\dedicatory{}

\author[J. Falc\'o]{Javier Falc\'o}
\address[Javier Falc\'o]{Departamento de An\'alisis Matem\'atico, Universidad de Valencia, Doctor Moliner 50, 46100 Burjasot (Valencia),
Spain.}
\email{Francisco.J.Falco@uv.es}

\author[P. M. Gauthier]{Paul M. Gauthier}
\address[Paul M. Gauthier]{D\'epartement de math\'ematiques et de statistique, Universit\'e de Montr\'eal, Montr\'eal, Qu\'ebec, Canada H3C3J7.}
\email{gauthier@dms.umontreal.ca}

\address{}
\email{}

\begin{document}

\begin{abstract}
In 1955, Lehto showed that, for every measurable function $\psi$ on the unit circle $\mathbb T,$ there is a function $f$ holomorphic in the unit disc, having $\psi$ as radial limit  a.e. on $\mathbb T.$ We consider an analogous problem for solutions $f$ of homogenous elliptic equations $Pf=0$ and, in particular, for holomorphic functions on Riemann surfaces and harmonic functions on Riemannian manifolds.
\end{abstract}


\thanks{First author was supported by MINECO and FEDER Project MTM2017-83262-C2-1-P. Second author was supported by NSERC (Canada) grant RGPIN-2016-04107.}

\maketitle

\section{Introduction}

In 1955, O. Lehto \cite{L} showed that given an arbitrary measurable function $\psi$ on the interval $[0,2\pi),$ there  exists a function $f$ holomorphic in the unit disc $\D\subset\C$ such that 
$$
	\lim_{\rho \to 1} f(\rho e^{i\theta}) = \psi(\theta), \quad \mbox{for} \quad a.e. \quad  \theta\in[0,\pi). 
$$
Lehto's theorem shows that the radial boundary values of  holomorphic functions in the unit disc can be prescribed almost everywhere on the boundary of the disc. On the other hand, any attempt to prescribe angular boundary values fails dramatically due to the  Luzin-Privalov uniqueness theorem \cite{No}. This result asserts that if a meromorphic function $f$  in the unit disc $\D$ has angular limit $0$ at each point of a subset of the boundary having positive linear measure, then $f= 0.$ 

For $p\in\C$ and $r>0,$ we denote by $B(p,r)$ the open disc of center $p$ and radius $r$. We will denote the Lebesgue $2$-measure by $m$. Our main result is the following. 

\begin{theorem}\label{main}
Given an arbitrary measurable function $\psi$ on the interval $[0,2\pi),$ 
whose restriction to some closed subset $S\subset[0,2\pi)$ is continuous,
there exists a  function $f$ holomorphic in  $\D,$  and 
for every $\theta\in S$ and a.e. $\theta\in [0,2\pi)$, there is a set $E_\theta$ closed in $\D,$ such that
$$
	f(z)\to \psi(\theta) \quad \mbox{as} \quad z\to e^{i\theta}, \, z\in E_{\theta}, 
$$
and 
$$
	\lim_{r\to 0}\frac{m\big(B(e^{i\theta},r)\cap E_\theta\big)}{m\big(B(e^{i\theta},r)\cap\D\big)}=1.
$$
\end{theorem}

Thus, although we may not prescribe angular approximation, we may prescribe approximation at almost every point of the boundary, from within  a set $E_{\theta}$ whose complement $\D\setminus E_{\theta}$ at every point of the unit circle is asymptotically negligible with respect to the Lebesgue measure. 

More generally, we shall present such a result for solutions of elliptic equations on manifolds. Our result applies in particular to harmonic functions on Riemannian manifolds and to holomorphic functions on Riemann surfaces. 

Let $M$ be an oriented real analytic manifold with countable base. We shall  denote by $*$ the ideal point of the one-point compactification $M^*$ of $M.$ Fix a distance function $d$ on $M$ and a positive Borel measure $\mu$ for which open sets have positive measure and compact sets have finite measure. Then  $\mu$ is regular (see \cite[Theorem 2.18]{R}). 
On $M$ the Lebesgue measure of a measurable set is not well-defined, but since $M$ is smooth, Lebesgue measure zero is invariant under change of coordinates, so the notion of absolute continuity of the measure $\mu$ (with respect to Lebesgue measure) is well-defined. We shall assume that our measure $\mu$ is absolutely continuous.
Let $U$  be an open subset of $M,$ $p$ a boundary point of $U$ and $F$ a closed subset of $U.$ For $\alpha\in[0,1]$, we shall say that the set $F\subset U$ has $\mu$-density $\alpha$ at $p$ relative to $U,$ if 
$$
	\mu_{U}(F,p):=\lim_{r\to 0}\frac{\mu\big( B(p,r)\cap F\big)}{\mu\big(B(p,r)\cap U\big)}=\alpha.
$$

Denote by $\vartheta$ the trivial vector bundle  $\vartheta=M\times\R^k.$ For a (Borel) measurable subset $S\subset M,$ denote by $\mathcal M(S,\vartheta),$ the family of  measurable sections of $\vartheta$ over $S.$ Thus, an element $u\in \mathcal M(S,\vartheta)$ can be identified with a $k$-tuple $u=(u_1,\ldots,u_k)$ of measurable functions $u_j:S\to\R, \, j=1,\ldots,k.$ For an open set $U\subset M,$ we denote by $C^\infty(U,\vartheta)$ the family of smooth sections on $U$ endowed with the topology of uniform convergence on compact subsets of all derivatives. For $u\in C^\infty(U,\vartheta)$ and $x\in U,$ we denote $|u(x)|=\max\{|u_1(x)|,\ldots,|u_k(x)|\}.$ 
Let $P:\vartheta\to \eta$ be an elliptic operator on $M$ with analytic coefficients, where
$\eta$ is a real analytic vector bundle on $M$ of the same rank $k.$ With this notation we have the following result.

\begin{theorem}\label{bundles}
Let $M, d, \mu, \vartheta, \eta, P$ be as above and suppose that $P$ annihilates constants. 
Let $U\subset M$ be  an arbitrary open subset  and  $\varphi\in\mathcal M(\partial U,\vartheta)$
an arbitrary Borel measurable section on the boundary $\partial U,$  
whose restriction to some closed subset $S\subset\partial U$ is continuous.  
Then, for  an arbitrary regular  $\sigma$-finite  Borel measure $\nu$ on $\partial U,$  
there exists $ \widetilde\varphi\in C^\infty(U,\vartheta)$ with $P \widetilde\varphi=0,$  such that, for $\nu$-almost every $p\in \partial U$, 
and for every $p\in S,$ $ \widetilde\varphi(x)\to \varphi(p)$, as $x\to p$ outside a set of $\mu$-density 0 at $p$ 
relative to $U.$  
\end{theorem}

Consider the two extremal situations, where $S$ is empty and $S$ is the entire boundary $\partial U$ respectively. If $S=\emptyset,$ then Theorem \ref{bundles} solves an asymptotic {\it measurable} first boundary value problem. If $S=\partial U,$ then Theorem \ref{bundles} solves an asymptotic {\it continuous} first boundary value problem.  The following two corollaries simply state that Theorem \ref{bundles} applies in particular for harmonic functions of several variables and to holomorphic functions of a single complex variable.

\begin{corollary}\label{harmonic}
Let $M$ be a Riemannian manifold and let $\mu$ be the associated volume measure on $M.$ 
Let $U\subset M$ be  an arbitrary open subset  and  $\varphi$ an arbitrary Borel measurable function $\varphi$  on the boundary $\partial U,$  whose restriction to some closed subset $S\subset\partial U$ is continuous. Then, for an arbitrary regular  $\sigma$-finite  Borel measure $\nu$ on $\partial U,$    
there exists a  harmonic function $\widetilde\varphi$ on $U,$  such that, for $\nu$-almost every $p\in \partial U$ and for every $p\in S,$
$\widetilde\varphi(x)\to \varphi(p)$, as $x\to p$ outside a set of $\mu$-density 0 at $p$ relative to $U.$  
\end{corollary}

\begin{corollary}\label{holomorphic}
Let $M$ be an open Riemann surface, $\pi:M\rightarrow\mathbb C$ a holomorphic immersion and  $\mu$ the associated measure on $M.$ Let $U\subset M$ be  an arbitrary open subset  and  $\varphi$ an arbitrary Borel measurable function $\varphi$  on the boundary $\partial U,$  
whose restriction to some closed subset $S\subset\partial U$ is continuous. 
Then, for an arbitrary regular  $\sigma$-finite  Borel measure $\nu$ on $\partial U,$    
there exists a  holomorphic function $\widetilde\varphi$ on $U,$  such that, for $\nu$-almost every $p\in \partial U$ and for every $p\in S,$
$\widetilde\varphi(x)\to \varphi(p)$, as $x\to p$ outside a set of $\mu$-density 0 at $p$ relative to $U.$  
\end{corollary}

\begin{proof}
Although the theorem is for real vector bundles and the $\overline\partial$-operator on a Riemann surface is generally considered as an operator between complex vector bundles of rank 1, we may also consider it as an operator between real vector bundles of rank 2 (see  \cite[Remark 3.10.10, Theorem 3.10.11]{Na}).
\end{proof}

{\em Remark.} Riemann surfaces are complex manifolds of dimension 1. 
Our proof does not allow us to prove an analogue of Corollary \ref{holomorphic} for holomorphic functions on higher dimensional  complex manifolds because the proof of Theorem \ref{bundles} is based on the Malgrange-Lax Theorem \cite{Na} which is for differential operators $P:\xi\to\eta$ between bundles of equal rank. On complex manifolds, the Cauchy-Riemann operator $\overline\partial$ maps forms of type $(p,q)$ to forms of type $(p,q+1).$ Thus, $\overline\partial:\mathcal E^{p,q}\to \mathcal E^{p,q+1}.$ For $\overline\partial:\mathcal E^{p,0}\to \mathcal E^{0,1},$ it is elliptic and in particular it is elliptic for the case   $\overline\partial:\mathcal E^{0,0}\to \mathcal E^{0,1},$ mapping functions to forms of type $(0,1).$ On a complex manifold of dimension $n,$ this is a map between bundles of respective (complex) ranks $1$ and $n$ (see \cite[3.10.10]{Na}). Thus, in order to have an operator between bundles of equal rank, we must restrict our attention to complex manifolds of dimension $1,$ that is, Riemann surfaces. 

When $S=\partial U,$  corollaries \ref{harmonic} and \ref{holomorphic} were proved in \cite{FG2} and \cite{FG1} respectively. 
When $M=\C, \, U=\D$ and $S=\emptyset,$ Corollary \ref{holomorphic} gives Theorem \ref{main}.


\section{Runge-Carleman approximation}

A closed subset $E$ of $M$ is said to satisfy the open $K-Q$-condition if, for every compact $K\subset M$ there is a compact $Q\subset M$ such that $K\subset Q^\circ$ and $E\cap Q$ is contained in $Q^\circ$.

An exhaustion $(K_j)_{j=1}^\infty$ of $M$ is said to be  regular if,  for each $n$, the sets $K_{n}$ are compact, $ K_{n}\subset K_{n+1}^\circ,$ $M^*\setminus K_n$ is connected and $M=\cup_{n=1}^{\infty} K_{n}^\circ$. We say that an exhaustion $(K_j)_{j=1}^\infty$ of $M$ is open compatible with a closed subset $E$ of $M$ if, for every $j=1,2,\ldots,$  $E\cap K_j$ is contained in $K_j^\circ.$

\begin{lemma}\label{compatible}
Let $E$ be a closed subset of $M,$ satisfying the  open $K-Q$-condition, then there exists a regular exhaustion of $M$ which is open compatible with $E.$ 
\end{lemma}

\begin{proof}
Let $(K_j)_{j=1}^\infty$ be a regular exhaustion of $M$.  
We shall define recursively an exhaustion  $(Q_{n})_{n=1}^{\infty}$  of $M$ with certain properties. From the $K-Q$ condition, we choose a compact set $Q_1,$ such that $K_1\subset Q_1^\circ$ and $E\cap Q_1\subset Q_1^\circ.$ Now, we may choose a compact set $Q_2,$ such that $K_1\cup Q_1\subset Q_2^\circ$ and $E\cap Q_2\subset Q_2^\circ.$ 
Suppose we have selected compact sets $Q_1,\ldots,Q_{n},$ such that $K_j\cup Q_j\subset Q_{j+1}^\circ$ and $E\cap Q_{j+1}\subset Q_{j+1}^\circ,$ for $j=1,\ldots,n-1.$ We may choose a compact set $Q_{n+1},$ such that $K_n\cup Q_n\subset Q_{n+1}^\circ$ and $E\cap Q_{n+1}\subset Q_{n+1}^\circ.$ Thus,  we have inductively constructed an exhaustion  $(Q_n)_{n=1}^\infty$ such that, for each $n,$ $K_n\cup Q_n\subset Q_{n+1}^\circ$ and $E\cap Q_n\subset Q_n^\circ.$  We denote by $Q_{n,*}^c$ the connected component of $M^*\setminus Q_n$ that contains the point $*$ and put $ L_n=M\setminus   Q_{n,*}^c$. Then,  $(L_{n})_{n=1}^{\infty}$ is a regular exhaustion of $M$ (see  \cite[p. 224]{Na}). Furthermore, for each $n,$ $E\cap\partial Q_n=\emptyset,$ so $E\cap L_n\subset L_n^\circ$. Thus, the exhaustion $(L_n)_{n=1}^\infty$ is  open compatible with $E.$ 
\end{proof}

A closed set $E\subset M$ is said to be a set of Runge-Carleman approximation, if for every open neighbourhood $U$ of $E,$ every  section $f\in C^\infty(U,\mathcal\vartheta),$ with $Pf=0,$ and every positive continuous function $\varepsilon$ on $E,$ there is a global section $u\in C^\infty(M,\mathcal\vartheta),$ with $Pu=0,$   such that $|u-f|<\varepsilon$ on $E.$

\begin{theorem}\label{Runge-Carleman}
Let $E$ be a closed subset of $M$  satisfyting the open $K-Q$ condition, with $M^*\setminus E$ connected.   Then $E$ is a set of Runge-Carleman approximation. 
\end{theorem}

\begin{proof}
By lemma \ref{compatible}, let $(L_{n})_{n=1}^{\infty}$ be a regular exhaustion of $M$ which is open compatible with $E$ and  set $L_0=\emptyset$.  Fix an open neighbourhood $U$ of $E,$ and a  section $f\in C^\infty(U,\mathcal\vartheta),$ with $Pf=0$. Consider  $\varepsilon$ a  continuous and positive function on $E,$ which we may assume is continuous and positive on all of $M$ and set 
$$
	\varepsilon_{n}=\min\{\varepsilon(x):x\in L_n\}>0, \quad n=1,2,\ldots.
$$

Now we construct recursively a sequence $(u_n)_{n=0}^\infty$ of sections $u_{n}\in C^\infty(M,\vartheta)$ 
such that $|u_{n}-u_{n-1}|<\varepsilon_{n}/2^{n}$ on $ L_{n-1}$  and $|u_{n}-f|<\varepsilon_{n}/2^{n}$ on $E\cap (L_{n}\setminus  L_{n-1})$. Consider $u_0=0$. For $n=1$ we only need to check the second condition on $u_1$ since the first condition is void. Note that $U_1=L_1^\circ\cap U$ is an open set containing $E\cap L_1$ such that $M\setminus U_1$ has no compact connected components and $Pf=0$ on $U_1$. By the Malgrange-Lax theorem, see \cite{Na}, there exists a section 
$u_{1}\in C^\infty(M,\vartheta)$ with $Pu_{1}=0,$ such that $|u_1-f|<\varepsilon_{1}/2$ on $E\cap  L_{1}.$

Assume now that we have fixed $(u_{n})_{n=1}^{N-1}$ satisfying the required conditions.
Consider two  open sets $V_N, W_N$ such that 
\begin{align*}
&E\cap (L_N\setminus L_{N-1})\subset V_N\subset (U\cap (L_N\setminus L_{N-1})),&\\
&L_{N-1}\subset W_N,&\\
&V_N\cap W_N=\emptyset.&
\end{align*}
Since $M^*\setminus E$ is connected and $(L_n)_{n=1}^\infty$ is a regular exhaustion, without loss of generality, we can assume that $M\setminus (V_N\cup W_N)$ has no compact connected components.

 Define $g\in C^\infty(G,\vartheta),$ by putting $g=u_{N-1}$ on $ W_{N}$ and $g=f$ on $V_{N}$. Set 
$$
	K=L_{N-1}\cup\left(E \cap (L_N\setminus L_{N-1})\right)
$$
and $U_N=(V_N\cup W_n)\cap U$. Then, $Pg=0$ on $U_N$ and $M\setminus U_N$ has no compact connected components. By the Malgrange-Lax theorem again, there exists a section $u_{N}\in C^\infty(M,\vartheta)$ with $Pu_{N}=0,$ such that 
$$
	\max_{x\in K}|u_{N}(x)-g(x)|< \frac{\varepsilon_N}{2^{N}}.
$$
The section $u_N$ has the required properties, which completes the inductive construction of the sequence $(u_n)_{n=0}^\infty.$  

For every $x\in M,$ the sequence $\{u_n(x)\}_{n=0}^\infty$ is Cauchy and hence $u$ converges pointwise to a section $\vartheta.$ Let $u(x)=\lim_{n\to\infty}u_{n}(x)$ for every $x\in M$. 
Since for every natural number $j,$ the sequence  $(u_{n})_{n=j}^{\infty}$ converges to $u$ in $C^\infty(L_j^\circ,\vartheta)$ and $Pu_{n}=0$ on $M,$ we have that $u\in C^\infty(L_j^\circ,\vartheta)$ and also $Pu=0$ on $L_j^\circ.$  Since this holds for every $j=1,2,\ldots,$ we have that $u\in C^\infty(M,\vartheta)$ and $Pu=0.$ 

To finish we show that $|u(x)-f(x)|\le \varepsilon(x)$ on $E$. Fix $x\in E.$ Then, there exists a unique natural number $n=n(x),$ such that $x\in L_n\setminus L_{n-1}.$ We have
\begin{align*}
|u(x)-f(x)|&\le |u(x)-u_n(x)|+|u_n(x)-f(x)|\\
&\le \left( \sum_{k=n}^\infty|u_{k+1}(x)-u_k(x)|\right)+\frac{\varepsilon_n}{2^n} \le 
\sum_{k=n}^\infty\frac{\varepsilon_k}{2^k} < 
\frac{\epsilon_n}{2^{n-1}} \le \frac{\varepsilon(x)}{2^{n-1}}\le \varepsilon(x).
\end{align*}	
\end{proof}

\begin{corollary}\label{locally finite}

Let $E$ be a subset of $M$ that is a union of a locally finite family of disjoint continua and suppose that  $M^*\setminus E$ is connected.  Then $E$ is a set of Runge-Carleman approximation. 
\end{corollary}

\begin{proof}
Notice that $E$ is closed, since it is the union of a locally finite family of closed sets. 
We only need to show that $E$ satisfies the open $K-Q$ condition. For this, fix a compact set $K$ in $M$. We denote the connected components of $E$ by $E_j$  and we may assume that they are ordered so that $E_1, \cdots, E_m$ are the connected components which meet $K.$ Set $L = K\cup E_1\cup\cdots\cup E_m$ and let $Q$ be a compact neighbourhood of $L$ disjoint from the closed set $E_{m+1}\cup E_{m+2}\cup\cdots.$ Then $Q$ satisfies the required conditions.  
\end{proof}


\section{Proof of Theorem \ref{bundles}}

\begin{lemma}\label{bundleQK}
Let $U$ be a proper open subset of a  manifold $M$ and $Q$ and  $K$ be disjoint   compact subsets of $\partial U.$ Then, for each $\varepsilon>0,$ there exists $\delta>0$ and an open set $V_\delta$ that is a $\delta$-neighbourhood 
 of $K$ in $M$ disjoint from $Q$ such that 
\[
	\frac{\mu\big(B(p,r)\cap U\cap V_\delta\big)}{\mu\big(B(p,r)\cap U\big)} < \varepsilon, \quad 
	\forall p\in Q, \quad \forall r>0.
\]

Furthermore,  $\delta$ can be chosen so that  $B(p,r)\cap V_\delta=\emptyset$ for all $p\in Q$ and $r<\delta$.
\end{lemma}

\begin{proof}
Set  $r_0=d(Q,K)/2>0$. We claim that 
\begin{equation}
\label{minmupositive}
	\rho:=\min_{p\in Q}\mu\big(B(p,r_0)\cap U\big)>0.
\end{equation}

Assume this were not the case and we have that $\rho=0$. Then, by the compactness of $Q$, we could find a sequence of points $(p_n)_{n=1}^\infty \subset Q$ convergent to a point $p_0\in Q$ so that 
$\mu\big(B(p_n,r_0)\cap U\big)<1/n$ and $d(p_n,p_0)<r_0/2$. But then we would have that $\mu\big(B(p_0,r_0/2)\cap U\big)\leq \mu\big(B(p_n,r_0)\cap U\big)<1/n$ for every natural number $n$. Hence $\mu\big(B(p_0,r_0/2)\cap U\big)=0$ contradicting the fact that $\mu$ has positive measure on open sets. Thus equation \eqref{minmupositive} holds.

Consider $V_\delta$ a $\delta$-neighbourhood  of $K$ in $M$ with $\delta<r_0$ and $\mu(U\cap V_\delta)<\varepsilon \rho$. 
It is clear that for such $\delta$  we have that $V_\delta$ and $Q$ are disjoint and  $B(p,r)\cap V_\delta=\emptyset,$ for all $p\in Q$ and all $r<r_0.$ Thus, the last statement of the lemma is obvious by choosing $\delta=r_0$.  Also, for any $r\geq r_0$ we have that 
\[
	\frac{\mu\big(B(p,r)\cap U\cap V_\delta\big)}{\mu\big(B(p,r)\cap U\big)}\leq  \frac{\mu(U\cap V_\delta)}{\mu\big(B(p,r)\cap U\big)} < \frac{\varepsilon\rho}{\mu\big(B(p,r)\cap U\big)}\leq \varepsilon.
\]

\end{proof}

The following lemma was stated in  \cite[Lemma 4]{FG2} for volume measure on a Riemannian manifold, but the same proof yields the following more general version. 

\begin{lemma}
\label{smallremoval}
Let $U$ be a proper open subset of a  manifold $M$ and $C$ a  connected  compact subset of $U$ with $\mu(C)=0.$ Then, for each   $\epsilon>0$  there is a connected open neighbourhood $R$ of $C$ in $U$ such that 
\[
	\frac{\mu\big(R\cap  B(p,r)\big)}{\mu\big(U\cap B(p,r)\big)} < \epsilon, 
\]
for every $p\in \partial U$  and every $r>0.$
\end{lemma}

We recall that an open subset $W$ of a real manifold $M$  is an {\it open  parametric ball} if there is a chart $\varphi:W\to\B,$ where $\B$ is an open ball in the  Euclidean space and $\varphi(W)=\B$. A subset $H\subset M$ is a {\it closed parametric ball}, if there is a parametric ball $\varphi:W\to\B$ and a closed ball $\overline B\subset\B,$ such that $H=\varphi^{-1}(\overline B).$

\begin{lemma}\label{continuous}
Under the hypotheses of Theorem \ref{bundles}, 
there exists a set $F,$ with $S\subset F\subset  \partial U$ and $\nu(\partial U\setminus F)=0,$ and 
$u\in C(U,\xi)$ on $U,$  such that, for every $p\in F$, $u(x)\to \varphi(p)$, as $x\to p$ outside a set of $\mu$-density 0 at $p$ relative to $U.$  
\end{lemma}

\begin{proof}
We start by showing that there exists a subset $F\subset\partial U$ containing $S$ of the form 
$$
	F=S\dot\cup\left(\dot\cup_{n=1}^\infty Q_n\right),
$$ 
with $Q_n$ compact so that the restriction of $\varphi$ is continuous on $Q_n$ and $\nu(\partial U\setminus F)=0$. 

First we assume that $\nu(\partial U\setminus S)<+\infty$. 
By Lusin's theorem (see \cite{H} and \cite[Theorem 2]{W}), there exists a compact set $Q_1$  in $\partial U\setminus S$ such that $\nu\left((\partial U\setminus S)\setminus Q_1\right)<2^{-1}$ and the restriction of $\varphi$ to $Q_1$ is continuous. Now, again by Lusin's theorem, we can find a compact set $Q_2$ in $(\partial U\setminus S)\setminus Q_1$ with $\nu(\big((\partial U\setminus S)\setminus Q_1)\setminus Q_2\big)<2^{-2}$ so that the restriction of $\varphi$ to $Q_1\dot\cup Q_2$ is continuous. By induction we can construct a sequence of compact sets $(Q_n)_{n=1}^\infty$ so that $Q_{n}\subset (\partial U\setminus S)\setminus \cup_{j=1}^{n-1} Q_j$, $\nu((\partial U\setminus S)\setminus \cup_{j=1}^{n} Q_j)<2^{-n}$ and 
the restriction of
$\varphi$ to $Q_1\dot\cup\cdots\dot\cup Q_n$ is continuous for $n=1,2,3,\ldots$. We set 
\[
	F=S\dot\cup\left(\dot\cup_{n=1}^\infty Q_n\right).
\] 
It is obvious that $(Q_n)_{n=1}^\infty$ is a family of pairwise disjoint compact sets and $\nu(\partial U\setminus F)=0$.
 
If $\nu(\partial U\setminus S)$ is not finite, by the $\sigma$-finiteness of the measure $\nu$, there exists a pairwise disjoint sequence of measurable sets $(R_l)_{l=1}^\infty$ with $\nu(R_l)<+\infty$ and 
$\partial U\setminus S=\dot{\bigcup}_{l=1}^\infty R_l$. By the previous argument applied to the section $\varphi$ restricted to the set $R_l$ we can find a pairwise disjoint sequence of compact sets $(Q_{n,l})_{n=1}^\infty$ of $R_l,$ so that the restriction of $\varphi$ is continuous on $Q_{n,l}$ and $\nu(R_l\setminus \dot\cup_{n=1}^\infty Q_{n,l})=0.$  Then,
\[
F=S\dot\cup\left(\dot\cup_{n=1}^\infty \dot\cup_{l=1}^\infty Q_{n,l}\right),
\]
satisfies the desired result.

We now begin to extend the section $\varphi$. For this we shall construct inductively a sequence of increasing sets $(E_n)_{n=1}^\infty$ and a sequence of sections $(f_n)_{n=1}^\infty$. We can write $S=\cup_{n=1}^\infty S_n$ and $F=S\dot\cup\left(\dot\cup_{n=1}^\infty Q_n\right)$, with $S_n$ and $Q_n$ compact, $S_n$ increasing and $Q_n$ pairwise disjoint, so that the restriction $\varphi_n$ of $\varphi$ to $F_n=S\cup Q_1\dot\cup\cdots\dot\cup Q_n$ is continuous.  By Lemma \ref{bundleQK}, for $l=2,3,\ldots,$ there is an open $\delta_{1,l}$-neighbourhood $V_{1,l}$ of $Q_l$ in $M$ such that
\begin{equation}
\label{neigVj}
\begin{aligned}
	\frac{\mu\big(B(p,r)\cap U\cap V_{1,l}\big)}{\mu\big(B(p,r)\cap U\big)} < \frac{1}{2^l}, \quad 
	\forall p\in Q_{1}\cup S_1, \quad \forall r\le 1,\\
\mu\big(B(p,r)\cap U\cap V_{1,l}\big)=0, \quad \forall p\in Q_1 \cup S_1, \quad \forall r<\delta_{1,l}.
\end{aligned}
\end{equation}
By the compactness of $Q_l$, the sets $V_{1,l}$ can be chosen to be a finite union of open parametric  balls in $M$. 
Let
\[
E_1=U\setminus \cup_{l=2}^\infty V_{1,l}.
\]
Since $S\cup Q_1$  is closed in $E_1\cup S\cup Q_1,$ by the Tietze extension theorem we can extend the section $\varphi_1$ to a continuous section $f_1$ on $E_1\cup S\cup Q_1.$

 Set $E_0=\emptyset$ and assume that for $j=1,\ldots,n,$ we have fixed positive constants $\delta_{j,l}<1/j,$ for $l\ge j+1,$ sets $E_j=U\setminus \cup_{l=j+1}^\infty V_{j,l}$ with $V_{j,l}$ being an open $\delta_{j,l}$-neighbourhood  of $Q_l$ in $M\setminus E_j$ that is  a finite union of open parametric  balls in $M$  and sections $f_j$ continuous on $E_j\cup F_j$  such that
$$
	f_j(x) = 
\left\{
	\begin{array}{ll}
		f_{j-1}(x),	&	x\in E_{j-1},\\
		\varphi_j(x)=\varphi(x),	&	x\in F_j,	
	\end{array}
\right.
$$
and 
\begin{equation}
\label{neighbours}
\begin{aligned}
	\frac{\mu\big(B(p,r)\cap U\cap V_{j,l}\big)}{\mu\big(B(p,r)\cap U\big)} < \frac{1}{2^l}, \quad 
	\forall p\in \cup_{k=1}^j(S_k\cup Q_k), \quad \forall r\le 1,\\
	\mu\big(B(p,r)\cap U\cap V_{j,l}\big)=0, \quad \forall p\in \cup_{k=1}^j(S_k\cup Q_k), \quad \forall r<\delta_{j,l},
\end{aligned}
\end{equation}
 for $j=1,\ldots,n$ and $l=j+1,j+2,\ldots$. For the step $n+1$, using   Lemma \ref{bundleQK} again we have that for every natural number $l>n+1$  there is an open $\delta_{n+1,l}$-neighbourhood $V_{n+1,l}$ of $Q_l$ in $M\setminus E_n$ such that
\begin{equation*}
\begin{aligned}
	\frac{\mu\big(B(p,r)\cap U\cap V_{n+1,l}\big)}{\mu\big(B(p,r)\cap U\big)} < \frac{1}{2^l}, \quad 
	\forall p\in \cup_{k=1}^{n+1}(S_k\cup Q_k), \quad \forall r\le 1,\\
	 \mu\big(B(p,r)\cap U\cap V_{n+1,l}\big)=0, \quad \forall p\in \cup_{k=1}^{n+1}(S_k\cup Q_k), \quad \forall r<\delta_{n+1,l}.
\end{aligned}
 \end{equation*}
Without loss of generality we can assume that $\delta_{n+1,l}<1/(n+1)$,   and, by the compactness of $Q_l$, the sets $V_{n+1,l}$ can be chosen to be a finite union of open parametric balls in $M$. 
Set
\[
E_{n+1}=U\setminus \cup_{l=n+2}^\infty V_{n+1,l}.
\]

Note that $E_n\cup F_{n+1}$ is relatively closed in $E_{n+1}\cup F_{n+1}$. Furthermore, the section $f_{n+1}$ defined as $f_{n+1}=f_{n}$ on $E_n$ and  $f_{n+1}=\varphi_{n+1}=\varphi$ on $F_{n+1}$ is continuous on the set $E_n\cup F_{n+1}$ since $E_n\cap V_{n,n+1}=\emptyset$. Therefore, by the Tietze extension theorem we can extend the section $f_{n+1}$ to a continuous section on $E_{n+1}\cup F_{n+1}$ that we denote in the same way. 

Note also that $\cup_{n=1}^\infty E_n=U$. Indeed, if $x\in U$, since $U$ is open there exists $r_x>0$ so that $B(x,r_x)\subset U$. Fix a natural number $l_0$ so that $\frac{1}{l_0}<r_x$. Then, for every $l>l_0$, since $V_{n,l}$ is a  $\frac{1}{l}$-neighbourhood  of $Q_l\subset \partial U$ in $M$, we have that $x\notin V_{n,l}$ for every natural number $n$. Thus, for $n\geq l_0$ we have that $x\in E_{n}=U\setminus \cup_{l=n+1}^\infty V_{n,l}$.  Then, the section $u,$ defined on $U$ as $u(x)=f_n(x)$ if $x\in E_n,$ is continuous at $x.$ Since $x$ is arbitrary we have the $u$ is continuous on $U$. 

There only remains to show that,
for every $p\in F$, $u(x)\to \varphi(p)$, as $x\to p$ outside a set of $\mu$-density 0 at $p$ relative to $U.$ 
For this, fix $p\in F.$ Then we can find a natural number $n$ so that $p\in F_n$. Note that $f_n$ is continuous on $E_n\cup F_n$ and we have defined  $u=f_n$ on $E_n$ and $f_n=\varphi_n=\varphi$ on $F_n.$  Therefore, for every $p\in F$, $u(x)\to \varphi(p)$ as $x\to p$ in $E_n.$ By construction, $U\setminus E_n$ has $\mu$-density $0$ at $p$  relative to $U.$ 
\end{proof}

Before we continue, we introduce some terminology.  A compact subset $K\subset M$ is a parametric Mergelyan set if there is an open parametric ball 
$\varphi:W\to\B$, with $K\subset W$, and a compact set $Q\subset \B,$ such that $\B\setminus Q$ is connected and $K=\varphi^{-1}(Q).$ 
A subset $E$ of a manifold $M$ is a 
{\it Mergelyan chaplet}, which we simply call a chaplet, if it is the countable disjoint union of a (possibly infinite)
locally finite family $E_{j}$ of pairwise disjoint  parametric Mergelyan sets 
$E_j.$ We denote the chaplet by $E=(E_j)_j.$ By Corollary \ref{locally finite}, a chaplet is a Runge-Carleman set.

\begin{proof}[Proof of Theorem \ref{bundles}]
By Lemma \ref{continuous}, there exists a set $F,$ with $S\subset F\subset  \partial U$ and $\nu(\partial U\setminus F)=0,$ and 
$u\in C(U,\xi)$ on $U,$  such that, for every $p\in F$, $u(x)\to \varphi(p)$, as $x\to p$ outside a set of $\mu$-density 0 at $p$ relative to $U.$  

Let $\mathcal S=\{S_l\}_{l=1}^{\infty}$ be a locally finite family of smoothly bounded compact parametric balls $S_l$ in $U$ such that $U=\cup_lS_l^0$ and $|S_l|<dist(S_l,\partial U),$ where $|S_l|$ denotes the diameter of $S_l$.  Assume also that none of these balls contains another.  
We may also assume that the balls  become smaller as we approach $\partial U$ so that the oscillation $\omega_l=\omega_l(u)$ of $u$ on $S_l$ is less than $1/l,$ for each $l.$ 
Let $s_l=\partial S_l$ for $l\in \mathbb N$.   Since $\mu$ is absolutely continuous,  
Lemma \ref{smallremoval}  tells us that
there is an open neighbourhood $R_{j}$ of $s_j$ in $U$ such that 
\begin{equation}
\label{Rj}
	\frac{\mu\big(R_{j}\cap  B(p,r)\big)}{\mu\big(U\cap B(p,r)\big)} < \frac{1}{2^j}, \quad \forall \, p\in  \partial U, \quad \forall \, r>0. 
\end{equation}

Without loss of generality we may assume that each $R_{j}$ is a smoothly bounded shell. That is, that in a local coordinate system, $R_{j}=\{x: 0<\rho_{j}<\|x\|<1\}$. By the local finiteness of $\mathcal S$ we may also assume that if $s_j\cap s_l=\emptyset$, then $R_j\cap R_l=\emptyset$.

Consider the closed set $A=U\setminus \bigcup_{k=1}^\infty R_{k}$. Then, denoting by $H_j=S_j^0\cap A$ and $A_j=H_j\setminus\bigcup_{k=1}^{j-1}S_k^0$ we have that 
\[
	A = \bigcup_{j=1}^\infty( S_j^0\cap A) = 
	\bigcup_{j=1}^\infty H_j = \dot{\bigcup}_{j=1}^\infty\left( H_j\setminus\bigcup_{k=1}^{j-1}S_k^0\right) = \dot{\bigcup}_{j=1}^\infty A_j.
\]
For each $j,$ the set  $H_j$ is a parametric Mergelyan set in $S_j^0$  and the family $(H_j)_{j=1}^\infty$ is locally finite, but they may not be disjoint. However, the $(A_j)_{j=1}^\infty$ form a locally finite family of disjoint compacta and hence $A=(A_j)_{j=1}^\infty$ is a Mergelyan chaplet.

Let us fix now a continuous function $\varepsilon:A\to (0,1]$ so that $\varepsilon(x)\to 0,$  when $x\to\partial U$. Since $(A_j)_{j=1}^\infty$ is a locally finite family of compacta, we may construct a family $(V_j)_{j=1}^\infty$ of disjoint open neigbouhoods $A_j\subset V_j, \, j=1,2,\ldots.$
For each $A_j$ of $A$  we choose a point $x_{A_{j}}\in A_{j}$ and define a function $g$ on $V=\cup_jV_j$
as $g(x)=\sum_{j}u(x_{A_{j}})\chi_{V_{j}}(x).$ Since the function $g$ is constant in each connected component $V_j$ and $P$ annihilates constants we have that $Pg=0$ on $V$.

We claim that $U^*\setminus A=\cup_kR_k$ is connected. Choose some $R_j$ and let $\mathcal R_j$ be the connected component of $\cup_kR_k$ containing $R_j.$ Then $\mathcal R_j$ is the union of a subfamily of $(R_k)_{k=1}^\infty.$ Let us show that this subfamily connects $R_j$ with $*$ in $U^*$. Suppose this is not the case. Conside the sets $V=\cup_{s_l\subset\mathcal R_j}  S_l^{\circ}$ and $W=\cup_{s_l\not\subset\mathcal R_j}  S_l^{\circ}$. Both sets are open because they are the union of open sets.  Note  that if $s_k\cap s_l=\emptyset$, then $R_k\cap R_l=\emptyset.$  
 Now, for every set $s_l$, either $s_l$ intersects some $s_k\subset \mathcal R_j$ or it is disjoint from every $s_k\subset \mathcal R_j,$ in which case $R_\ell$ is disjoint from every $R_k\subset \mathcal R_j.$ In the second case,  $R_\ell$ cannot be in the bounded complementary component of any $R_k$ with $R_k\subset\mathcal R_j$, for then $S_\ell$ would be a subset of $S_k^0$ which is forbidden. Therefore $R_\ell$ and consequently $S_\ell^0$ lies in the unbounded complementary component of every $R_k$ with $R_k\subset\mathcal R_j$. This means that $S_ \ell\cap S_k=\emptyset.$ We have shown that, if $s_\ell\not\subset \mathcal R_j,$ then $S_\ell\cap S_k=\emptyset,$ for every $s_k\subset\mathcal R_j$ and consequently $V\cap W=\emptyset.$
If, as we supposed, $\mathcal R_j$ is bounded in $U,$ both $V$ and $W$ are non-empty and this  contradicts the assumption that $U$ is connected. Thus, every $\mathcal R_j$ is  unbounded in $U.$ Since $U^*\setminus A=\cup_j(\mathcal R_j\cup\{*\})$ is the union of a family of connected sets having point $*$ in common, it follows that $U^*\setminus A$ is connected as claimed.

By Corollary \ref{locally finite}, there exists a  function   $\widetilde\varphi \in C^\infty(U,\vartheta)$ with $P\widetilde\varphi=0,$ such that $|\widetilde\varphi-g|<\varepsilon$ on $A$.
We show now that 
\begin{equation}\label{approach in A}
	|\widetilde\varphi(x)-u(x)|\to 0, \quad \mbox{as} \quad x\to p
	\in \partial U, \quad x\in A.
\end{equation}

If $(x_n)_{n=1}^\infty$ is a sequence of points in $A$ tending to $p\in \partial U$, then $(x_{A_{j_n}})_{n=1}^\infty$ is also a sequence of points in $A$ tending to $p\in \partial U$, where $x_{A_{j_n}}$ is the previously fixed point in the  $A_{j_n}$ of  $A$ containing $x_n$. Indeed this follows automatically since
\[
d(x_{A_{j_n}},x_n)\leq \vert A_{j_n}\vert\leq d(A_{j_n},\partial U)\leq d(x_n,\partial U)\to 0,
\]
when $n$ goes to infinity.

Also, 
	\begin{align*}
		\limsup_{n\to\infty}\vert \widetilde \varphi(x_{n})-u(x_n)\vert
& \leq\limsup_{n\to\infty}\big(\vert  \widetilde \varphi(x_{n}) -g(x_n)\vert+\vert g(x_{n})-u(x_n)\vert\big)\\
& \leq\limsup_{n\to\infty}\big(\varepsilon(x_{n})+\vert u(x_{A_{j_{n}}})-u(x_n)\vert\big)\\
& \leq \limsup_{n\to\infty}\big(\varepsilon(x_{n})+ \omega_{j_n}(u)\big)
		=0.
	\end{align*}

We now show that $A$ satisfies that, 
\begin{equation}\label{density A}
	\mu_{U}(A,p)=\liminf_{r\rightarrow 0}\frac{\mu\big(B(p,r)\cap A\big)}{\mu\big(B(p,r)\cap U\big)}=1.
\end{equation}
For this we shall show that
$$
	\limsup_{r\rightarrow 0}\frac{\mu(B(p,r)\cap(U\setminus A))}{\mu(B(p,r)\cap U)}=0.
$$
For fixed $\varepsilon>0$ consider $j_\varepsilon$ so that
$$
	\sum_{j\ge j_\varepsilon} 2^{-j}<\varepsilon.
$$

Consider $r_\varepsilon>0$  so that $B(p,r_\varepsilon)$ is disjoint from the  neigbourhoods $R_{j}$ of the sets $s_j$ for $j\le j_\varepsilon$.  Then, for all $r<r_\varepsilon$, 
 since $U\setminus A = \cup_j R_{j}$, we have that
\begin{align*}
	\hspace{1cm}&\hspace{-1cm}\frac{\mu(B(p,r)\cap(U\setminus A))}{\mu(B(p,r)\cap U)}
	 = \frac{\mu(B(p,r) \cap\left( \cup_j R_{j}\right))}{\mu(B(p,r)\cap U)}\\
	 &\le\sum_{j>j_\varepsilon}\frac{\mu(B(p,r)\cap R_{j})}{\mu(B(p,r)\cap U)}\\
	 &\le \sum_{j>j_\varepsilon}2^{-j} <\varepsilon.\hspace{2cm}(\text{by } \eqref{Rj})
	\end{align*}
	Thus, the $\mu$-density of $U\setminus A$  relative to $U$ at $p$ is at most $\varepsilon.$ Since $p$ and $\varepsilon$ are arbitrary, this proves 
(\ref{density A}). 

	Note that the function $u$ has all of the properties desired in the theorem, except that of satisfying the differential equation $Pu=0.$ The function $\widetilde\varphi$ does satisfy the equation $P\widetilde\varphi=0$ and also satisfies the desired properties, because of (\ref{approach in A}) and (\ref{density A}).

\end{proof}

\end{document}